\documentclass[11pt]{article}

\usepackage{color}
\usepackage[T1]{fontenc}
\usepackage[utf8]{inputenc} 
\usepackage{amsmath,amssymb,amsbsy,amsthm}

\newtheorem{myproposition}{Proposition}[section]
\newtheorem{mytheorem}[myproposition]{Theorem}

\newtheorem{myconjecture}[myproposition]{Conjecture}
\newtheorem{mycorollary}[myproposition]{Corollary}

\newtheorem{myobservation}[myproposition]{Observation}

\newtheorem{myproblem}[myproposition]{Problem}
\usepackage{amsfonts}
\usepackage{amsmath}
\def\gr{\mathcal{G}}
\def\zet{\mathbb{Z}}
\def\har{{\rm{har}}}

\makeatletter
\def\imod#1{\allowbreak\mkern10mu({\operator@font mod}\,\,#1)}
\makeatother

\begin{document}

\title{Note on the group edge irregularity strength of graphs}


\author{Marcin Anholcer$^1$, Sylwia Cichacz$^{2,3}$\\
$^1$Pozna\'n University of Economics and Business\\
$^2$AGH University of Science and Technology Krak\'ow, Poland\\
$^3$ University of Primorska,  Koper, Slovenia}



\maketitle

%



\begin{abstract}
We investigate the \textit{edge group irregularity strength} ($es_g(G)$) of graphs, i.e. the smallest value of $s$ such that taking any Abelian group $\gr$ of order $s$, there exists a function $f:V(G)\rightarrow \gr$ such that the sums of vertex labels at every edge are distinct.  In this note we provide some upper bounds on $es_g(G)$ as well as for edge  irregularity strength $es(G)$ and harmonious order $\har(G)$.
\end{abstract}

\section{Introduction}

In 1988  Chartrand et al. \cite{ref_ChaJacLehOelRuiSab1} proposed the problem of irregular labeling. This problem was motivated by the well
known fact that a simple graph of order at least 2 must contain a pair of vertices with the same degree.  The situation changes if we consider multigraphs. Each multiple edge may be represented with some integer label and the (\textit{weighted}) degree of any vertex $x$ is then calculated as the sum of labels of all the edges incident to $x$. The maximum label $s$ is called the \textit{strength} of the labeling. The labeling itself is called \textit{irregular} if the weighted degrees of \textbf{all} vertices are distinct. The smallest value of $s$ that allows an irregular labeling is called the \textit{irregularity strength of} $G$ and denoted  by $s(G)$. This problem was one of the major sources of inspiration in graph theory \cite{AhmAlMBac,ref_AigTri,ref_AmaTog,ref_AnhCic1,ref_BacJenMilRya,ref_FerGouKarPfe,ref_KalKarPfe1,ref_KarLucTho,ref_MajPrz,ref_Nie,ref_ThoWuZha,refXuLiGe}. For example the concept of $\gr$-irregular labeling is a generalization of irregular labeling on Abelian groups.  The \textit{group irregularity strength} of $G$, denoted $s_g(G)$, is the smallest integer $s$ such that for every Abelian group $\gr$ of order $s$ there exists a $\gr$-irregular labeling $f$ of $G$. The following theorem, determining the value of $s_g(G)$ for every connected graph $G$ of order $n\geq 3$, was proved by Anholcer, Cichacz and Milani\v{c} \cite{ref_AnhCic1}.

\begin{mytheorem}[\cite{ref_AnhCic1}]\label{AnhCic1}
Let $G$ be an arbitrary connected graph of order $n\geq 3$. Then
$$
s_g(G)=\begin{cases}
n+2,&\text{if   } G\cong K_{1,3^{2q+1}-2} \text{  for some integer   }q\geq 1,\\
n+1,&\text{if   } n\equiv 2 \imod 4 \wedge G\not\cong K_{1,3^{2q+1}-2} \text{  for any integer   }q\geq 1,\\
n,&\text{otherwise.}
\end{cases}
$$
\end{mytheorem}

The notion of \textit{edge irregularity strength} was defined by Ahmad, Al-Mushayt and Ba\v{c}a \cite{AhmAlMBac}. 
  The weight of an edge $xy$ in $G$ is the sum of the labels of its end vertices $x$ and $y$. A vertex $k$-labeling $\phi\colon V(G)\rightarrow \{1,2,\ldots,k\}$ is called edge irregular $k$-labeling of $G$ if every two different edges have different weights. The minimum $k$ for which $G$ has an edge irregular $k$-labeling is called the \textit{edge irregularity strength} of $G$ and denoted by $es(G)$. They established the exact value of the edge irregularity strength of paths, stars, double stars and Cartesian product of two paths. They also gave a lower bound for $es(G)$.
In the literature there is also investigated the total version of this concept, namely \textit{edge irregular total labeling} \cite{ref_BacJenMilRya,refXuLiGe}.

Graham and~Sloane defined \textit{harmonious labeling} as a direct extension of additive bases and graceful labeling. We call a labeling $f \colon V (G) \rightarrow\zet_{|E(G)|}$ harmonious if it is an injection such that each edge  $xy\in E(G)$ has different sum $f (x) + f (y)$ modulo $|E(G)|$. When
$G$ is a tree, exactly one label may be used on two vertices. They conjectured that any tree is harmonious (the conjecture is still unsolved) \cite{ref_GraSlo}. Beals et al. (see \cite{ref_BeaGalHeaJun}) considered the concept of harmoniousness with respect to arbitrary Abelian groups. \.Zak in \cite{ref_Zak} generalized the problem and introduced a new parameter, the \textit{harmonious order of $G$}, defined as the smallest number $t$ such that there exists an injection $f:V(G)\rightarrow \mathbb{Z}_t$ (or surjection if $t<V(G)$) that produces distinct edge weights.

The problem of harmonious order is connected with a~problem of sets in Abelian groups with distinct sums of pairs.
A subset $S$ of an Abelian group $\Gamma$, where $|S|=k$, is an $S_2$-set of size $k$ if all the sums of 2 different elements in $S$ are distinct in $\Gamma$. Let $s(G)$ denote the cardinality of the largest $S_2$-set in $\Gamma$. Two central functions in the study
of $S_2$-sets are $v(k)$ and $v_{\gamma}(k)$, which give the order of the smallest Abelian and cyclic group $G$,
respectively, for which $s(G)\geq k$. Since cyclic groups are a special case of Abelian groups, clearly
$v(k)\leq v_{\gamma}(k)$, and any upper bound on $v_{\gamma}(k)$ is also an upper bound on $v(k)$ \cite{Haa}. Note that $har(K_n)=v_{\gamma}(n)\leq n^2+O(n^{36/23})$ \cite{ref_Zak}.

Recently Montgomery, Pokrovskiy and  Sudakov proved the following theorem.
\begin{mytheorem}[\cite{ref_Sudakov}]
Every tree $T$ of order $n$ has an injective $\Gamma$-harmonious labeling for any Abelian group $\Gamma$ of order $n + o(n)$.
\end{mytheorem}

In this paper we would like to introduce a new concept which gathers the ideas of $\gr$-irregular labeling, edge irregularity strength and harmonious order, namely \textit{group edge irregularity strength}. 

Assign an element of an Abelian group $\gr$ of order $s$ to every vertex $v \in V(G)$. For every edge $e=uv \in E(G)$ the \textit{weight} is defined as:
\begin{eqnarray*}
wd(uv)=w(u)+w(v).
\end{eqnarray*}
The labeling $w$ is $\gr$-edge irregular if for $e\neq f$ we have $wd(e) \neq wd(f)$. \textit{Group edge irregularity strength} $es_g(G)$ is the lowest $s$ such that for every Abelian group $\gr$ of order $s$ there exists $\gr$-edge irregular labeling of $G$.

\section{Bounds on $es_g(G)$}

Let us start with general lower bound on $es_g(G)$. Of course, the order of the group must be equal at least to the number of edges of $G$.

\begin{myobservation}\label{lemma_cycle_below0}
For each graph $G$ with $|E(G)|=m$, $es_g(G)\geq m$.
\end{myobservation}

The above bound can be improved e.g. for cycles.

\begin{myproposition}\label{lemma_cycle_below}
If $n \equiv 2 \imod 4$, then $es_g(C_n) \geq n+1$.
\end{myproposition}

\begin{proof}
Assume we can use some $\gr$ of order $2(2k+1)$. Obviously $\gr=\mathbb{Z}_2\times \gr_1$ for some group $\gr_1$ of order $2k+1$. There are $2k+1$ elements $(1,a)$ where $a\in\gr_1$ and all of them have to appear as the edge weights, so
$$
\sum_{e\in E(G)}{wd(e)}=(1,b_1)
$$
For some $b_1\in \gr_1$. On the other hand 
$$
\sum_{e\in E(G)}{wd(e)}=2\sum_{v\in V(G)}{w(v)}=(0,b_2)
$$
for some $b_2\in\gr_1$, a contradiction.
\end{proof}

Note that each $\gr$-edge irregular labeling of $K_n$ has to be injection, what implies that $\har(K_n)\leq es_g(K_n)$. So  $es_g(K_n)={n \choose 2}$ only for $n\leqslant 3$ \cite{ref_GraSlo}. Recall that for   an integer $n$ that has all primes distinct in its factorization, there exists a unique Abelian group of order $n$, namely $\zet_n$. Therefore we obtain that $es_g(K_5)=11$, $es_g(K_6)=19$, $es_g(K_{12})=114$, $es_g(K_{14})=178$ and $es_g(K_{15})=183$ \cite{ref_Zak}. Although it is not known what are the exact values of $\har(K_n)$ for $n\geq16$, the lower bound   $n^2-3n\leq \har(K_n)$ is given for each $n \geq 3$ \cite{ref_Zak}, hence we easily obtain the following observation.
\begin{myobservation} If $n \geq 3$, then  $n^2-3n \leq es_g(K_{n})$.
\end{myobservation}

We give now several upper bounds on $es_g(G)$


In  \cite{AhmAlMBac} Ahmad, Al-Mushayt and Ba\v{c}a obtained exponential upper bound on $es(G)$ depending on Fibonacci number with seed values $F_1=1, F_2=2$. However, because of $es(G)\leq har(K_n)$ we obtain the following.
\begin{myproposition}
If $G$ is a graph of order $n$, then $es(G)\leq n^2+O(n^{36/23})$.
\end{myproposition}
Let $|E(G)|=m$. Note that in general we do not know whether $es_g(G) \leq \har(K_n)$, however we are able to show the following.

\begin{mytheorem}\label{marcin}
For each graph $G$, $es(G) \leq es_g(G)\leq  p(2es(G)) \leq p(2 \har(G))$, where $p(k)$ is the least prime greater than $k$.
\end{mytheorem}

\begin{proof}
The first inequality follows from the fact that if we can find a $\gr$-irregular labeling for any group $\gr$ with a given order $p$, then in particular we can do it for the cyclic group and if it distinguishes the  weights modulo $p$, then also the labeling with integers (where we use $p$ instead of $0$ as a label) generates an irregular labeling.

For every prime $p$ there is only one Abelian group $\gr$ of order $p$, namely $\zet_p$. If in the labeling one uses only the labels less than $p/2$, the addition inside the group is equivalent to the addition in $\zet$, so the second inequality follows.

The last inequality is implied by the fact that $es(G)\leq \har(G)$.
\end{proof}

From the Bertrand–Chebyshev theorem \cite{ref_The} it follows that for any positive integer $i$ there exists a prime number between $i$ and $2i$, which easily leads to the following.

\begin{mycorollary}
Let $G$ be a graph of order $n$. Then $es_g(G)\leq 4n^2 +O(n^{36/23}).$
\end{mycorollary}

However, for larger $n$, better bounds are known. For example, Naguro \cite{ref_Nag} proved that there is a prime between $i$ and $1.2i$ provided that $i\geq 25$. This gives us the following upper bound. 

\begin{mycorollary}
Let $G$ be a graph of order $n\geq 25$. Then $es_g(G)\leq 2.4n^2 +O(n^{36/23}).$
\end{mycorollary}

Recently Baker, Harman and Pintz \cite{ref_BakHarPin} proved that for sufficiently large $i$, there is a prime between $i$ and $i+i^{0.525}$. This allows to obtain the following result for large graphs.

\begin{mycorollary}\label{corMA}
Let $G$ be a graph of sufficiently large order $n$. Then $es_g(G)\leq 2n^2 +O(n^{36/23}).$
\end{mycorollary}

The latter result can be reduced for some special classes of graphs. First, let us consider a graph $G$ of the following form. Assume that the vertices of $G$ can be divided into four sets $V_{11}$, $V_{12}$, $V_{21}$ and $V_{22}$ such that for every edge $xy$, we have that $x\in V_{ij}$ and $y\in V_{kl}$ implies $i=k$ or $j=l=1$. Moreover, assume that $|V_{11}|+|V_{12}|\leq \left\lceil n/2 \right\rceil$, $|V_{11}|+|V_{21}|\leq \left\lceil n/2 \right\rceil$ and $|V_{21}|+|V_{22}|\leq \left\lceil n/2 \right\rceil$. A special case is a graph in which $|V_{11}|+|V_{21}|=0$, i.e. a graph with two components having orders $\left\lceil n/2 \right\rceil$ and $\left\lfloor n/2 \right\rfloor$.

\begin{mycorollary}
Let $G$ be a graph defined as above, where $n$ is sufficiently large. Then $es_g(G)\leq 1.5n^2 +O(n^{36/23})$.
\end{mycorollary}

\begin{proof}
For any graph of sufficiently large order $\left\lceil n/2 \right\rceil$, there is a group of prime order $g=0.5n^2 +O(n^{36/23})$ that allows a group edge irregular labeling of this graph by Corollary~\ref{corMA}. Of course, $g$ is not divisible by $3$.

Let us start with modifying the graph $G$ by adding some edges so that the subgraphs with vertex sets $V_{11}\cup V_{12}$, $V_{11}\cup V_{21}$ and $V_{21}\cup V_{22}$ become complete graphs.

Let us take any group $\gr$ of order $3g$. Of course such group must be of the form $\zet_3\times \gr^\prime$ for some group $\gr^\prime$ of order $g$. We label the vertices of $G$ with elements $(g_1,g_2)$ of $\gr$, where $g_1\in \zet_3$ and $g_2\in \gr^\prime$ in the following way. First we choose $g_2$ for the vertices in $V_{11}\cup V_{21}$ in such way that the edges are distinguished even if all $g_1$ are equal (it is possible, since any graph of order $\left\lceil n/2 \right\rceil$ can be labeled with a group of order $g$). Then we do the same with vertices of $V_{11}\cup V_{12}$, by labeling the vertices of $V_{12}$ with the elements of $\gr^\prime$ not used in $V_{11}$. Finally we label the vertices of $V_{22}$ with the elements not used in $V_{21}$. This distinguishes the edges inside $V_{11}\cup V_{12}$ and inside $V_{21}\cup V_{22}$, no matter what are the values of $g_1$. Now we choose $g_1=0$ for the vertices in $V_{11}\cup V_{12}$ and $g_1=1$ for the vertices in $V_{21}\cup V_{22}$, which distinguishes the edges from different sets: the first coordinate of every edge inside $V_{11}\cup V_{12}$ is now $0$, inside $V_{21}\cup V_{22}$ equals to $2$ and for each edge between $V_{11}$ and $V_{12}$ it equals to $1$. This means that the labeling is the group edge irregular labeling of the modified graph, so also of each of its subgraphs, in particular of $G$. Thus
$$es_g(G)\leq 3g\leq 1.5n^2 +O(n^{36/23})$$.
\end{proof}

A similar reasoning allows as to strengthen the result for graphs having more than two components of almost the same order. As we know, in any group of odd (in particular prime) order $p$, $x=y$ if and only if $2x=2y$, provided that $x,y\neq 0$. Thus if one uses different value of $g_1$ in every component, it is enough to distinguish the edges inside every component by the elements $g_2$ of the subgroup $\gr^\prime$ of $\gr=\zet_p\times \gr^\prime$ (of course $p$ must be prime and $|\gr^\prime|$ not divisible by $p$ if we want this decomposition to be unique). It gives us the following result.

\begin{mycorollary}
Let $G$ be a graph of order $n$, consisting of $q\geq 2$ components with orders different by at most $1$, where $n$ is sufficiently large. Let $p$ be the smallest odd prime not less than $q$. Then
$$es_g(G)\leq \frac{2p}{q^2}n^2 +O(n^{36/23}).$$
\end{mycorollary}

Note that if also $q$ is sufficiently large, then we obtain $es_g(G)\leq 2n^2/q +O(n^{36/23}).$

\begin{myproposition}\label{forest}
For each forest $F$, $es_g(F)=m$. Moreover, any weighting of edges is possible for arbitrary choice of the label of one vertex in each component.
\end{myproposition}
\begin{proof}
Given any edge that is still not weighted, if one of the vertices has label $a$, and the edge is supposed to be weighted with $b$, it is enough to put $b-a$ on the other vertex.
\end{proof}

The notion of coloring number of a graph was introduced by  Erd\H{o}s and Hajnal in \cite{ErdHaj}. For a given graph $G$ by ${\rm col}(G)$ we denote its coloring number, that is the least integer $k$ such that each subgraph of $G$ has minimum degree less than $k$. Equivalently, it is the smallest $k$ for which we may linearly order all vertices of $G$ into a sequence $v_1,v_2,\ldots,v_n$ so that every vertex $v_i$ has at most $k-1$ neighbors preceding it in the sequence.
Hence $\chi(G)\leq {\rm col}(G)\leq \Delta(G)+1$.
Note that ${\rm col}(G)$ equals the degeneracy of $G$ plus $1$, and thus the result below may 
be formulated 
in terms of either of the two graph invariants.\\


\begin{mytheorem}\label{col_upper}
For every graph $G=(V,E)$, there exists a $\gr$-edge irregular labeling for any Abelian group $\gr$ of order $|\gr|\geq ({\rm col}(G)-1)(|E|-1)+1$.
\end{mytheorem}

\begin{proof} 

By Proposition \ref{forest} we can assume that $G$ is not a forest.
Fix any Abelian group $\gr$ of order $|\gr|\geq ({\rm col}(G)-1)(|E(G)|-1)+1$.
Let $v_1,v_2,\ldots,v_n$ be the ordering of $V(G)$ witnessing the value of ${\rm col}(G)$. We start with putting arbitrary color on $v_1$. Then we will color the remaining vertices of $G$ with elements of $\gr$ in $n-1$ stages, each corresponding to a consecutive vertex from among $v_2,v_3,\ldots,v_n$.
Initially no vertex except $v_1$ is colored. 
Then at each stage $i$, $i=2,3,\ldots,n$, 
we color the vertex $v_i$.
We will choose a color avoiding sum conflicts between already analyzed vertices and so that at all times the partial edge coloring has the desired property.

Namely we choose a color $w(v_i)\in \gr$ so that $wd(v_iv_j)$ for $j<i$ and $v_iv_j\in E(G)$  is distinguished from any $wd(v_tv_k)$ where $1\leq t\leq k$ and $v_tv_k\in E(G)$. Thus  we cannot use at most $({\rm col}(G)-1)(|E(G)|-1)$ colors.
\end{proof}
We immediately obtain the following result.

\begin{mycorollary}\label{nullcol}
For each graph $G$ of order at least $4$, $es_g(G)\leq  ({\rm col}(G)-1)(|E(G)|-1)+1$. 
\end{mycorollary}

Taking into account that for every planar graph $G$ we have ${\rm col}(G)\leq6$, we obtain the following corollary.

\begin{mycorollary}\label{Planar}
For each planar graph $G$ of order at least $4$, $es_g(G)\leq 5|E(G)|-4$. 
\end{mycorollary}

Note also that if we additionally want to have injection of colors on vertices, then within the proof of Theorem~\ref{col_upper} above, we obtain at most $n-1$ constraints while choosing a color for a given vertex. Consequently, by a straightforward adaptation 
of the proof above, we obtain the following.
\begin{mycorollary}\label{har}
For each graph $G$ of order at least $4$, $\har(G)\leq |V(G)|+ ({\rm col}(G)-1)(|E(G)|-1)$. 
\end{mycorollary}

The exact value of $es_g(C_n)$, where $C_n$ is a cycle of order $n$, is given by the following theorem.

\begin{mytheorem}
Let $C_n$ be arbitrary cycle of order $n\geq 3$. Then
$$
es_g(G)=\begin{cases}
n+1&\text{when   } n\equiv 2 \imod 4\\
n&\text{otherwise}
\end{cases}
$$
Moreover respective labeling exists for an arbitrary choice of the label of any vertex.
\end{mytheorem}

Remark: in fact, the labeling can be found for any group of order at least $es_g(C_n)$.

\begin{proof}
Labeling the vertices distinguishing the edge weights is in this case equivalent to the labeling of the edges distinguishing the vertex weighted degrees (we label the line graph, moreover $m$=$n$). Thus the theorem is a simple corollary of Theorem \ref{AnhCic1}.
\end{proof}

\begin{mytheorem}\label{dwudzielne}
Let $G=K_{m,n}$, then $es_g(G)=mn$.
\end{mytheorem}
\noindent\textbf{Proof.}
Let $\Gamma$ be an Abelian group of order $mn$. One of the consequences of the fundamental theorem of finite Abelian groups is that  for any divisor $k$ of $|\Gamma|$ there exists a subgroup $H$ of $\Gamma$ of order $k$. Therefore there exists $\Gamma_0< \Gamma$ such that $|\Gamma_0|=n$.
Let $V_1$ and $V_2$ be the partition sets of $G$ such that $|V_1|=m$ and $|V_2|=n$. Put all elements of $\Gamma_0$ on the vertices of the set $V_1$, whereas on the vertices of $V_2$ put all coset representatives for $\Gamma/\Gamma_0$. Note that all vertices incident with a vertex $v\in V_2$ obtain different weights which are elements of the coset $w(v)\Gamma_0$. Hence using   a coset decomposition of $\Gamma$ we are done.~\qed\\

From the above observation we obtain the following upper bound for bipartite graphs.
\begin{mycorollary}
Let $G$ be a bipartite graph of order $n$, then $es_g(G)\leq \left\lceil \frac{n^2-1}{4}\right\rceil$.
\end{mycorollary}
\noindent\textbf{Proof.}
Let $G$ have partition sets $V_1$ and $V_2$ of orders $n_1$ and $n-n_1$, respectively. Obviously $G$ is a subgraph of $K_{n_1,n-n_1}$, so by Theorem~\ref{dwudzielne} we obtain $es(G)\leq es(K_{n_1,n-n_1})= {n_1(n-n_1)}\leq  \left\lceil \frac{n^2-1}{4}\right\rceil$.~\qed\\

\section{Final remarks}

In the paper we presented a new graph invariant, the group edge irregularity strength $es_g(G)$. We presented the relations between that and other parameters, like harmonious order $\har(G)$. We also gave some lower and upper bounds for $es_g(G)$.

Based on them, we state the following conjecture.

\begin{myconjecture}
There exists a constant $c>0$, such that for every graph $G$ of size $m$, $es_g(G)\leq 2m+c$
\end{myconjecture}

Let us consider a version of the problem for directed graphs. Assume that the weight of an arc $(u,v)$ with tail $u$ and head $v$ is now computed as
\begin{eqnarray*}
wd((u,v))=w(u)-w(v).
\end{eqnarray*}

For example, if one considers directed acyclic graphs (DAGs), the following result analogous to Theorem \ref{col_upper} easily follows from the fact that the vertices of such digraph may be ordered so that each of them is preceded by all its in-neighbors (or all the out-neighbors).

\begin{myproposition}
Let $D$ be a DAG with $m$ arcs, maximum indegree $\Delta^-$ and maximum outdegree $\Delta^+$. Then $es_g(D)\leq (m-1)\min\{\Delta^-,\Delta^+\}+1$.
\end{myproposition}

Observe that directed acyclic graphs are connected with an old~problem of \textit{difference basis}  from the number theory  \cite{ref_RedRen}, therefore the following problem would be interesting.

\begin{myproblem}
Find the $es_g(D)$ for arbitrary digraph $D$.
\end{myproblem}

\nocite{*}
\bibliographystyle{amsplain}

\end{document}